\documentclass[11pt,a4paper]{amsart}
\usepackage{amsmath}
\usepackage{amssymb}
\usepackage[french,english]{babel}
\usepackage{amsfonts}
\usepackage{mathrsfs}
\usepackage{verbatim}
\usepackage{graphicx}
\usepackage{color}
\newtheorem{theo}{Theorem}
\newenvironment{Theo}{\begin{theo}\slshape}{\end{theo}}
\newtheorem{Lemm}{Lemma}[section]
\newtheorem{Prop}{Proposition}[section]
\newtheorem{Corol}{Corollary}[section]
\newtheorem{Exp}{Example}[section]
\newtheorem{rema}{Remark}
\newenvironment{Rema}{\begin{rema}\normalfont}{\end{rema}}
\newtheorem{defi}{Definition}

\newcommand{\rr}{\mathbf r}
\newcommand{\eps}{\varepsilon}

\def\qed{\hfill$\square$\par \vspace{.1cm}}
\newenvironment{Demo}{{\bf Proof.}}{\qed}
\newcommand{\R}{\mathbb{R}}

\newcommand{\N}{\mathbb{N}}
\newcommand{\Z}{\mathbb{Z}}

\newcommand{\Ec}{\mathcal{E}}

\newcommand{\Vc}{\mathcal{V}}

\DeclareMathOperator{\Ran}{Ran}
\newcommand{\norm}[1]{\left\Vert#1\right\Vert}

\newcommand{\tir}{\discretionary{.}{}{---\kern.7em}}

\newcommand{\di}{\mathrm d} %

\renewcommand{\Im}{\operatorname{Im}}

\newcommand{\Ker}{\operatorname{Ker}}

\usepackage{graphicx}

\newcommand{\cref}[1]{Corollary~\ref{#1}}
\usepackage{times}

\begin{document}

\title[A GRAPH WITHOUT ZERO IN ITS SPECTRA]
{A GRAPH WITHOUT ZERO IN ITS SPECTRA}

\author[C. Ann\'e]{Colette Ann\'e}
\address{CNRS/Nantes Universit\'e, Laboratoire de Math\'ematique Jean Leray, Facult\'e des Sciences, BP 92208, 44322 Nantes, (France).}

\email{colette.anne@univ-nantes.fr}

\author[H. Ayadi]{Hela Ayadi}

\address{Universit\'e de Monastir, (LR/18ES15)
\& Acad\'emie Navale, 7050 Menzel-Bourguiba (Tunisie)}
\email{ halaayadi@yahoo.fr }

\author[M. Balti]{Marwa Balti}

\address{King Faisel University}
\email{balti-marwa@hotmail.fr}

\author[N. Torki-Hamza]{Nabila Torki-Hamza}
\address{Universit\'e de Monastir, LR/18ES15 \& Institut Sup\'erieur d'Informatique de Mahdia (ISIMa) B.P 05, Campus Universitaire de Mahdia; 5147-Mahdia (Tunisie).}
\email{natorki@gmail.com}


\date{Version of \today, {\em file} : 2023V10.tex}

\begin{abstract}
In this paper we consider the discrete Laplacian acting on 1-forms and we study its spectrum relative to the spectrum of the 0-form Laplacian. We show that the non zero spectrum can coincide for these Laplacians with the same nature. We examine the characteristics of 0-spectrum of the 1-form Laplacian compared to the cycles of graphs.
\end{abstract}


\maketitle

\tableofcontents

\section*{Introduction}

\medskip

Spectral graph theory stands for an active area of research. Recently, many authors are interested in studying the spectra of the Laplacians on infinite graphs and the relationship to the structure and characteristic properties of graphs and operators, we can cite \cite{yc}, \cite{jor}, \cite{F}, \cite{gry}, \cite{gry1},\cite{lp}, \cite{moh}, \cite{so}, \cite{SW91}, \cite{St10}, \cite{nth},  and \cite{wo}. This question has been studied in several topics, for example, harmonic analysis on graphs, probability theory on Markov chains, potential theory, and so on...  In 1996, John Lott adressed the following question \cite{Jo96}: let M be a complete oriented Riemannian n-manifold, then the Laplacian $\Delta_p$ is an essentially 
self-adjoint non negative operator on square-integrable p-forms, $0\leq p\leq n$, 
and thus has a spectrum which is a subset of $\R^+;$ 
must zero always belong to the spectrum of some $\Delta_p$~?
In 2001, Farber and Weinberger \cite{FW} disproved the conjecture "zero-in-the-spectrum" for manifolds. They show that for any $n\geq 6$ there exists a closed manifold $M^n$ such that, for any degree,  zero does not belong to the spectrum of the Hodge-Laplace operator acting
on $L^2$ differential forms of the universal covering $\tilde M$.
In the discrete case,  Ayadi \cite{Ay} showed in 2017 that 0 is either in the spectrum of
the Laplacian on 0-forms, or in the spectrum of the Laplacian on 1-forms on a weighted  infinite
normalized graph such that the weight on edges is bounded from above and from below.
{\bf In this paper we generalize this result.} 
Indeed, we study how the spectrum of 0-form and 1-form Laplacians coincide outside zero. And in order to
characterize the spectrum outside zero, we study the conjecture "zero-in-the-spectrum" how
this is related with the structure of the graph.
 We define two Laplacians, mentioned
in \cite{A}, \cite{Ay} and \cite{BGJ}, one as an operator acting on functions on vertices denoted by $\Delta_0$ and the other one acting on skew-symmetric functions on edges denoted by $\Delta_1$.
We discuss the relation between
\begin{itemize}
  \item The eigenvalues of $\Delta_0$ and these of $\Delta_1$ counting their multiplicities for the finite case.
  \item  The spectrum of $\Delta_0$ and that of $\Delta_1$ for the infinite case.
  \item The essential and the discrete spectrum of $\Delta_0$ and $\Delta_1$.
  \item The problem of 0 in the spectrum.
\end{itemize}
After Section 1 devoted to the the preliminaries , we have three sections: 
In Section \ref{first}, we consider a finite connected graph $G$. We introduce the adjacency matrix associated to the 1-form Laplacian $\Delta_1$ such that it is easier to
examine its spectrum. We study its eigenvalues.
Although its relation to that of 0-forms $\Delta_0$. We can see that it reflects in a very natural way the structure of the graph, in particular the aspects related to the cycles of the graph. We show that the eigenvalues are the same outside of zero and they have the same multiplicity.
In section \ref{infinite}, we extended this result to infinite graphs. We explore in the case
of unbounded 1-form Laplacian, Weyl's criterion.
Section \ref{exp} presents some examples which support  our results.

\section{Preliminaries}

\subsection{Definition and notation}

\begin{itemize}
\item A graph $G$ is a couple $(\mathcal{V}, \mathcal{E})$ where $\mathcal{V}$ is a set at most countable whose elements are called vertices and $\mathcal{E}$ is a set of oriented edges, considered as a subset of $\mathcal{V}\times\mathcal{V}$.\\
 \item  If the graph $G$ has a finite set of vertices, it is called a finite graph. Otherwise, $G$ is called an infinite graph.\\
\item We assume that $\mathcal{E}$ has no self-loops and is symmetric:
$$v \in \mathcal{V} \Rightarrow (v, v) \notin \mathcal{E},
 \;\; (v_{1}, v_{2})\in \mathcal{E} \Rightarrow (v_{2}, v_{1})\in \mathcal{E}.$$
\item  Choosing an orientation of $G$ consists of defining a partition of $\mathcal{E}$: $\mathcal{E}^{+}\sqcup \mathcal{E}^{-}=\mathcal{E}$
 $$(v_{1}, v_{2})\in \mathcal{E}^{+} \Leftrightarrow (v_{2},v_{1})\in \mathcal{E}^{-}.$$
\item For $e=(v_{1}, v_{2})$, we denote
$$e^{-}=v_{1},\; e^{+}=v_{2}\text{ and } -e=(v_{2}, v_{1}).$$
\item We write $v_{1}\sim v_{2}$ for $e=(v_{1}, v_{2}) \in \mathcal{E}$.\\
\item The graph $G$ is connected if any two vertices $x$, $y$ in $\mathcal{V}$ can be joined by a path of edges $\gamma_{xy}$, that means $\gamma_{xy}=\{e_{k}\}_{k=1,..., n}$ such that
     $$ e^{-}_{1}=x,\; e^{+}_{n}=y\text{ and, if } n\geq2\; ,\;\forall j \; ;1\leq j \leq (n-1) \Rightarrow e^{+}_{j}=e^{-}_{j+1}.$$
\item A cycle of $G$ is a path of edges $\gamma_{xy}=\{e_{k}\}_{k=1,..., n}$ satisfying more over
     $$  e_n^+=e_1^-.$$
 \item  The degree (or valence) of a vertex $x$ is the number of edges emanating from $x$. We denote
$$ \mathrm{deg}(x):= \sharp\{  e\in \mathcal{E};\;e^{-} = x  \}.$$
  \item 
The graph $G$ is called a locally finite graph, if all its vertices are of finite valence. 

 \end{itemize}

 \subsection{Weighted graphs}
\begin{defi}\label{weight}
A weighted graph $(G, m, c)$ is given by a graph
 $G=(\mathcal{V}, \mathcal{E})$, a weight on the vertices
 $m:\mathcal{V} \rightarrow ]0, \infty[$ and a weight on the edges
  $c: \mathcal{E} \rightarrow ]0, \infty[$ such that 
 \begin{itemize}
\item $c(x, y)> 0,\; \forall (x, y) \in \mathcal{E}.$
 \item $c(x, y)=c(y, x),\; \forall (x, y) \in \mathcal{E}.$
 \end{itemize} 
\end{defi}

\textbf{Examples:}
- An infinite electrical network is a weighted graph $(G,m ,c)$ where the weights of the edges
correspond to the conductances $c$; their inverses are the resistances. And the weights of the vertices are given by $m(x)=\sum_{y\in \mathcal{V}} \frac{1}{ c(x, y)}< \infty,\; \forall x \in \mathcal{V}$, \textit{that means the vertices have no charge}.\\

-The graph $G$ is called a simple graph if the weights of the edges and the vertices equal $1$.\\

\textbf{All the graphs we shall consider in the sequel will be connected, locally finite and
weighted as given in Definition 1.}

\subsection{Functional spaces}
We denote the set of real functions on $\mathcal{V}$ by:
$$\mathcal{C}(\mathcal{V})=\{f:\mathcal{V} \rightarrow \R\}$$

and the set of functions of finite support by $\mathcal{C}_{0}(\mathcal{V})$.\\

Moreover, we denote the set of real skew-symmetric functions (or 1-forms) on $\mathcal{E}$ by:
$$\mathcal{C}^{a}(\mathcal{E})=\{\varphi:\mathcal{E} \rightarrow \R \;;
 \varphi(-e)=- \varphi(e)\}$$

and the set of functions of finite support by $\mathcal{C}^{a}_{0}(\mathcal{E})$.\\

We define on the weighted graph $(G, m, c)$ the following function spaces endowed with the scalar products.

\begin{itemize}
  \item[a)] $$l^{2}(\mathcal{V}):=\left\{ f\in \mathcal{C}(\mathcal{V})  ;\;
   \sum_{x\in \mathcal{V}} {m}  (x) f^{2}(x)<\infty \right\},$$

with the inner product
$$\langle f, g\rangle_{\mathcal{V}}=\sum_{x\in \mathcal{V}} {m}(x) f(x)g(x)$$

and the norm
$$\norm{f}_{\mathcal{V}}=\sqrt{\langle f, f \rangle_{\mathcal{V}}}.$$

\item[b)] $$l^{2}(\mathcal{E}):=\left\{ \varphi \in \mathcal{C}^{a}(\mathcal{E});\;
\frac{1}{2}\sum_{e\in \mathcal{E}} c(e) \varphi^{2}(e)<\infty \right\},$$

with the inner product
$$\langle \varphi, \psi \rangle_{\mathcal{E}}=
\frac{1}{2} \sum_{e\in \mathcal{E}} c(e) \varphi(e)\psi (e)$$

and the norm
$$\norm{\varphi}_{\mathcal{E}}=
\sqrt{\langle \varphi, \varphi\rangle_{\mathcal{E}}}.$$
\end{itemize}

Then, $l^{2}(\mathcal{V})$ and $l^{2}(\mathcal{E})$ are separable Hilbert spaces (since $\mathcal{V}$ is countable).

\subsection{Operators and properties}

\underline{The difference operator} 
$$\mathrm{d}:\mathcal{C}_{0}(\mathcal{V})\longrightarrow \mathcal{C}^{a}_{0}(\mathcal{E}),$$

is given by
$$\mathrm{d}(f)(e)=f(e^{+})-f(e^{-}).$$

\underline{The coboundary operator} is $\delta$, the formal adjoint of $\mathrm{d}$. Thus it satisfies
\begin{equation}\label{op div}
\langle \mathrm{d}f, \varphi \rangle_{\mathcal{E}}=
\langle f, \delta \varphi \rangle_{\mathcal{V}}\end{equation}
for all $ f \in \mathcal{C}_{0}(\mathcal{V})$ and for all $ \varphi \in \mathcal{C}^{a}_{0}(\mathcal{E}).$\\

As consequence, we have the following formula characterizing $\delta$, given by \cite{AT}:

\begin{Lemm}\label{lem1}
The coboundary operator $\delta$ is characterized by the formula

$$\delta \varphi (x)=\frac{1}{{m}(x)}  \sum_{e, e^{+}=x} c(e) \varphi (e),$$

for all $ \varphi \in \mathcal{C}^{a}_{0}(\mathcal{E}).$

\end{Lemm}

\begin{defi}
The Laplacian on 0-forms $\Delta_{0}$ defined  by $\delta \mathrm{d} $ on $\mathcal{C}_{0}(\mathcal{V})$ is given by
$$ \Delta_{0}f(x)=\frac{1}{{m}(x)}  \sum_{ y\sim x} c(x, y) \left(f (x)-f (y)\right).$$
In fact, we have
\begin{eqnarray*}
\Delta_{0}f(x)&=&\delta (\mathrm{d}f) (x)\\
&=&\frac{1}{{m}(x)}  \sum_{e, e^{+}=x} c(e) \mathrm{d}f (e)\\
&=&\frac{1}{{m}(x)}  \sum_{e, e^{+}=x} c(e) \left(f (e^{+})-f (e^{-})\right)\\
&=&\frac{1}{{m}(x)}  \sum_{y\sim x} c(x, y) \left(f (x)-f (y)\right).
\end{eqnarray*}
\end{defi}
\begin{defi}\label{lap1}
 The Laplacian on 1-forms $\Delta_{1}$ defined by $\mathrm{d}\delta$ on $\mathcal{C}^{a}_{0}(\mathcal{E})$ is given by
 $$\Delta_{1}\varphi(e)=\frac{1}{{m} (e^{+})}  \sum_{e_{1}, e_{1}^{+}=e^{+}} c(e_{1}) \varphi (e_{1})-\frac{1}{{m} (e^{-})}  \sum_{e_{2}, e_{2}^{+}=e^{-}} c(e_{2}) \varphi (e_{2}).$$
In fact, we have 
\begin{eqnarray*}
\Delta_{1}\varphi(e)&=& \mathrm{d}(\delta \varphi) (e)\\
&=& \delta \varphi(e^{+})-\delta \varphi(e^{-})\\
&=& \frac{1}{{m}(e^{+})}  \sum_{e_{1}, e_{1}^{+}=e^{+}} c(e_{1}) \varphi (e_{1})-\frac{1}{{m}(e^{-})}  \sum_{e_{2}, e_{2}^{+}=e^{-}} c(e_{2}) \varphi (e_{2}).
\end{eqnarray*}
\end{defi}
\begin{rema}
\begin{enumerate}
\item As the graph $G$ is connected and locally finite the operators $\mathrm{d}$ and $\delta$ are
closable, see \cite{AT}.
\item The operator $\Delta_{0}$ is a positif symmetric operator on $\mathcal{C}_{0}(\mathcal{V})$, as we have $\langle \Delta_{0}f, f\rangle_{\mathcal{V}}=\langle \mathrm{d} f, \mathrm{d} f\rangle_{\mathcal{E}},\;\forall f \in \mathcal{C}_{0}(\mathcal{V})$ and $\langle \Delta_{0}f, g\rangle_{\mathcal{V}}=\langle f, \Delta_{0}g\rangle_{\mathcal{V}},\;\forall f,\;g \in \mathcal{C}_{0}(\mathcal{V})$.
\item The operator $\Delta_{1}$ is a positif symmetric operator on $\mathcal{C}^{a}_{0}(\mathcal{E})$, as we have $\langle \Delta_{1}\varphi, \varphi\rangle_{\mathcal{E}}=\langle \mathrm{\delta} \varphi, \mathrm{\delta} \varphi \rangle_{\mathcal{V}},\;\forall \varphi \in \mathcal{C}^{a}_{0}(\mathcal{E})$ and $\langle \Delta_{1}\varphi, \psi \rangle_{\mathcal{E}}=\langle \varphi, \Delta_{1}\psi \rangle_{\mathcal{E}},\;\forall \varphi,\;\psi \in \mathcal{C}^{a}_{0}(\mathcal{E})$.
\end{enumerate}
\end{rema}
\section{Finite case}\label{first}
In this section, first we will consider a finite simple graph $G$, that means a graph with a finite number of vertices and edges and the weights $m=c=1$. We denote $k$ the number of vertices and $n$ the number of edges on $G$. Then we can write the Laplacian on 1-forms $\Delta_1$ as a matrix to find results in its 0-spectrum. After that we will generalize the result for weighted finite graph.\\

Now we give the matrix associated to the Laplacian on 1-forms. We have 
\begin{eqnarray*}
\Delta_{1}\varphi(e)&=& \mathrm{d}(\delta \varphi) (e)\\
&=& \delta \varphi(e^{+})-\delta \varphi(e^{-})\\
&=&  \sum_{e_{1}, e_{1}^{+}=e^{+}} \varphi (e_{1}) +\sum_{e_{2}, e_{2}^{+}=e^{-}}  \varphi (e_{2})\\
&=&  2 \varphi(e)+ \sum_{x, x\neq e^{-}} \varphi (x,e^{+}) +\sum_{y, y\neq e^{+}}  \varphi (y,e^{-})
\end{eqnarray*}
Then we obtain
$$ \Delta_{1}=2 \hbox{I}+\hbox{A}_{1}$$
where $\hbox{I}$ is the identity matrix of size n and $\hbox{A}_{1}$ is the adjacency matrix on 1-forms. \\
More precisely, we have $\hbox {A}_1=(a_{i,j})_{1\leq i,j\leq n}\;i \neq j$ and $e_{i}$ and $e_{j}$ are two edges with commun vertex, such that

$$
a_{i,j}=\left\{
  \begin{array}{ll}
    1\;\;\;\;\;\;\;\hbox{ if } e_{j}^{\pm}=e_{i}^{\pm}\\
    \\
   -1\;\;\;\;\hbox{ if } e_{j}^{\pm}=e_{i}^{\mp}\\
    \\
   0\;\;\;\;\;\;\;\hbox{ otherwise }.
\end{array}
\right.$$
\medskip
\bigskip
\\
\textbf{Examples:}
\begin{enumerate}
\item We consider the complete graph $K_5$ with five vertices and twenty edges. But the skew-symmetric functions are defined by ten values. Then the size of the matrix associated to the operator $\Delta_1$ is $10\times 10$.
 \begin{figure}[ht]
\includegraphics[height=4cm,width=5cm]{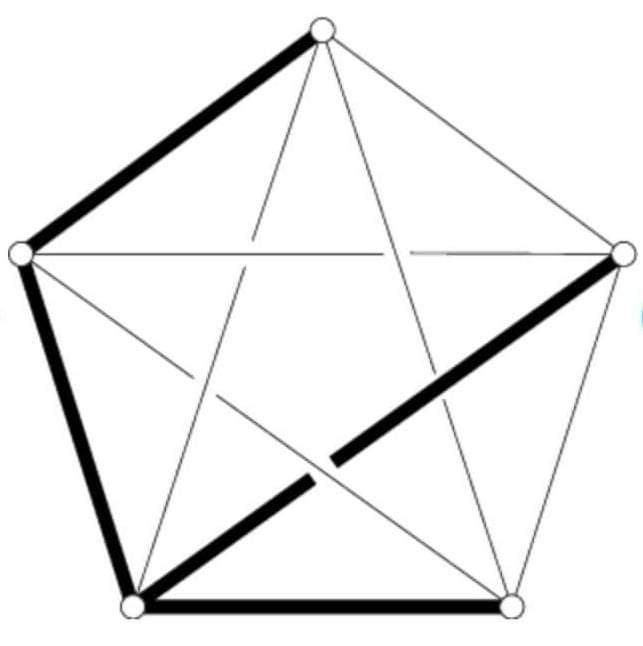}
\caption{The complete graph $K_5$}
 \end{figure}
$$\Delta_1=\left(
      \begin{array}{ccccccccccc}
        2 & -1 & 0 & 0 & -1 & 1 & 1 & -1 & -1 & 0\\
        -1 & 2 & -1 & 0 & 0 & 1 & 0 & 1 & 1 & -1\\
        0 & -1 & 2 & -1 & 0 & -1 & 1 & 1 & 0 & 1\\
        0 & 0 &  -1 & 2 & -1& 0& -1 & -1 & 1 & 1 \\ 
        -1& 0 & 0 & -1 & 2 & -1 & -1 & 0 & -1 & -1\\
        1 & 1 & -1 & 0 & -1 & 2 & 1 & 0 & 0 & -1\\
        1 & 0 & 1 & -1 & -1& 1 & 2 & 1 & 0 & 0\\
        -1& 1 & 1 & -1 & 0 & 0 & 1 & 2 & 1 & 0\\
        -1 & 1 & 0 & 1 & -1 & 0 & 0 & 1 & 2 & 1\\
        0 & -1 & 1 & 1 & -1& -1 & 0 & 0 & 1 & 2\\
      \end{array}
    \right)
$$
Then we have the eigenvalues of $\Delta_1$ equal to 5 with multiplicity 4 and \textit{0 with multiplicity 6}. And its known that the eigenvalues of $\Delta_0$ equal to 5 with multiplicity 4 and \textit{0 with multiplicity 1}. So we have the same non-zero eigenvalues with same multiplicity and the zero eigenvalue with different multiplicity.\\

\item Let $G$ a graph with 4 vertices and 4 edges then we have the matrix on 1-forms given by
$$\Delta_1=\left(
      \begin{array}{cccc}
        2 & -1 & 0 & -1\\
        -1 & 2 & -1 & 0 \\
        0 & -1 & 2 & -1 \\
       -1 & 0 &  -1 & 2  \\ 
        \end{array}
    \right)
$$
Hence we obtain the eigenvalues of $\Delta_1$ equal to 2 with multiplicity 2, 4 with multiplicity 1 and \textit{0 with multiplicity 1}, which is the same spectrum of  $\Delta_0$.
\end{enumerate}
We remark that the spectrum of $\Delta_0$ and $\Delta_1$ coincide in certain cases, indeed the
difference occur only for the value 0. So in the next section we will study the spectrum question of the different operators.
\subsection{Spectrum} Denote by $\sigma_p(A)$ the point spectrum of a linear operator $A$
which consists of eigenvalues. The spectrum of $\Delta_0$ and $\Delta_1$ respectively consist of non negative eigenvalues. We denote them $0=\mu_1\leq \mu_2 \leq . . . \leq \mu_m$ and $\lambda_1 \leq \lambda_2 \leq . . . \leq \lambda_n$ repeated according to their multiplicity.\\

\begin{Prop}\label{mult} The operators $\Delta_0$ and $\Delta_1$ have the same set of positive
  eigenvalues.
\end{Prop}
\begin{Demo} 
  Let $\mu$ be a non zero eigenvalue of $\Delta_0$ and $f$ an associated non constant eigenfunction, we have if $f\in \mathcal{C}_0(\mathcal{V})$, then $\mathrm{ d }f\in \mathcal{C}_0(\mathcal{E})$,
  $\|\mathrm{ d }f\|^2=\mu\|f\|^2\neq 0$ and
\begin{align*}
\Delta_1\mathrm{ d }f&=\mathrm{ d }\delta \mathrm{ d }f\\
&=\mathrm{ d }\mu f\\
&=\mu \mathrm{ d }f,
\end{align*}
therefore $\mathrm{ d }f$ is an eigenfunction of $\Delta_1$, and $\mu\in \sigma_p(\Delta_1)$.\\
For the spectrum of $\Delta_1$ we show in the same way that if $\varphi$ is an
eigenfunction 1-form of $\Delta_1$, then $\delta \varphi$ is the eigenfunction associated
to $\Delta_0$.
\end{Demo}
\begin{Prop}\label{multp}
The multiplicity of the non-zero eigenvalues of $\Delta_0$ and $\Delta_1$ are the same. 
\end{Prop}
\begin{Demo} 
Let $\lambda$ be a nonzero eigenvalue of multiplicity $m_0$ for $\Delta_0$ and multiplicity $m_1$
for $\Delta_1$. Then there exists a family of associated eigenfunctions $f_1,f_2,..,f_{m_0}$
linearly independent. So we obtain a family of eigenfunctions $(\mathrm{ d }f_i)_{1\leq i \leq m_0}$ associated
to the eigenvalue $\lambda$ for the operator $\Delta_1$. Let $(\alpha_i)_{1\leq i \leq m_0}$ be real numbers such that $\alpha_1 \mathrm{ d }f_1+...+\alpha_{m_0}\mathrm{ d }f_{m_0}=0$, then 
\begin{align*}
\delta(\alpha_1 \mathrm{ d }f_1+...+\alpha_{m_0}\mathrm{ d }f_{m_0})=&\lambda(\alpha_1 f_1+...+\alpha_{m_0}f_{m_0})\\
=&0
\end{align*}
but $\lambda\neq 0$ therefore $\alpha_i=0$ for all $1\leq i \leq {m_0}$
because the $f_i$ are linearly independent. This implies
that $\mathrm{ d }f_1,\mathrm{ d }f_2,..,\mathrm{ d }f_{m_0}$ are linearly independent. Hence, $m_1\geq m_0$. \\
For the second inequality, we changes the roles of $\mathrm{ d }$ and $\delta$.
\end{Demo}
\begin{Corol}
  $$\sigma_p(\Delta_0)\setminus\{0\}=\sigma_p(\Delta_1)\setminus\{0\}.$$
\end{Corol}
\begin{rema}
 The multiplicity of $ 0\in \sigma(\Delta_0)$ coincides with the number of connected components of the graph see \cite{moh91}.
\end{rema}
\begin{Prop}\label{zero}
If $G$ contains a cycle, then $0\in\sigma(\Delta_1)$.
\end{Prop}

\begin{Demo}\label{fct}
 For the second fact, let $C_\ell=(e_1,...,e_\ell)$ be a cycle in $G$, we must prove that there
 exists an eigenfunction $\varphi$ corresponding to $0$ for $\Delta_1$.  We define a
  skewsymmetric function $\varphi$ on $\mathcal{C}_{0}(\mathcal{E})$ by 
$$
\varphi(e)=\left\{
  \begin{array}{ll}
    1\;\;\;\;\;\;\;\hbox{ if } e \in C_\ell\\
    \\
   0\;\;\;\;\;\;\;\hbox{ otherwise }.
  \end{array}.
  \right.$$
We discuss three cases
\begin{itemize}
\item If $e\in C_\ell$, then there exists $i \in \lbrace 1,...,\ell\rbrace$ such that $e=e_i$ and
\begin{align*}
\Delta_1\varphi(e)&=\Delta_1\varphi(e_i)\\
&=2\varphi(e_i)-\varphi(e_{i+1})-\varphi(e_{i-1})\\
&=2-1-1=0
\end{align*}
\item If $e\notin C_\ell$ and related to an edge $e_i\in C_\ell$ means $e^\pm=e_i^\pm$, so we have
\begin{align*}
\Delta_1 \varphi(e)&=2\varphi(e)+\varphi(e_i)-\varphi(e_{i-1})\\
&=0+1-1=0
\end{align*}
\item If $e$ is not related to any edge of $C_\ell$, we obtain $\Delta_1\varphi(e)=0$.
\end{itemize}
\end{Demo}
\begin{defi}
\emph{The circuit rank}, denoted $\rr $, is the minimum number of edges that must be removed from the graph to break all its cycles, making it into a tree or a forest.
\end{defi}
\begin{Rema}
 A spanning tree $T$ of a connected graph $G$ contains $(k-1)$ edges ( where $k$ is the number of vertices and $n$ is the number of edges on $G$). Therefore, the number of edges you need to delete on $G$ in order to get a spanning tree equals to $n-(k-1)$, which is \emph{the circuit rank } of $G$. Also \emph{the circuit rank } $\rr $ is equal to the number of independent cycles in the graph, that is 
the size of a cycle basis (see \cite{Berge}).
\end{Rema}
\begin{Prop}If the graph is connected, then
the multiplicity of $0$ in the spectrum of $\Delta_1$ equals $\rr$.
\end{Prop}
\begin{Demo}
  From Proposition \ref{mult}, the non zero eigenvalues of $\Delta_0$ and $\Delta_1$ coincide with the same multiplicity. But by hypothesis $0$ is an eigenvalue of $\Delta_0$ of multiplicity $1$,
  so $\Delta_0$ has $k-1$ non zero eigenvalues, and $\Delta_1$ as well. Hence the multiplicity of
  $0$ in the spectrum of $\Delta_1$ equal $n-(k-1)=\rr$.
\end{Demo}
\begin{Corol}
Let $G$ be a connected tree, then $0\notin \sigma(\Delta_1)$. 
\end{Corol}
\subsection{Example}
Let us consider the complete graph $K_5$, we have $\sigma(\Delta_1)=\{0,5\}$, with $0$ is of multiplicity $6$ and the circuit rank of $K_5$ is $r=6$. 

\subsection{Weighted finite graph}
Let $(G,m,c)$ be a weighted finite graph with $k$ vertices and $n$ edges. All the previous
results can be generalized.\\
 We recall from Definition $\ref{lap1}$, the Laplacian on 1-forms $\Delta_{1}$ defined by $\mathrm{d}\delta$ on $\mathcal{C}_{0}(\mathcal{E})$ is given by
\begin{eqnarray*}
\Delta_{1}\varphi(e)&=& \mathrm{d}(\delta \varphi) (e)\\
&=& \delta \varphi(e^{+})-\delta \varphi(e^{-})\\
&=& \frac{1}{{m}(e^{+})}  \sum_{e_{1}, e_{1}^{+}=e^{+}} c(e_{1}) \varphi (e_{1})-\frac{1}{{m}(e^{-})}  \sum_{e_{2}, e_{2}^{+}=e^{-}} c(e_{2}) \varphi (e_{2}).
\end{eqnarray*}

Then we obtain
$$ \Delta_{1}=\hbox {I} +\hbox {A}_{1}$$
where $\hbox {I}$ is the diagonal matrix of size n and $\hbox {A}_{1}$ is the adjacency matrix on 1-forms. \\
More precisely, we have $\hbox {I}=(a_{i,i})_{1\leq i\leq n}=\left( c(e_i)(\frac{1}{m(e_{i}^{+})}+\frac{1}{m(e_{i}^{-})})\right) _{1\leq i\leq n}$ and $\hbox {A}_1=(a_{i,j})_{1\leq i,j\leq n}\;i \neq j$ and $e_{i}$ and $e_{j}$ are two edges \emph{with commun vertex}, such that

$$
a_{i,j}=\left\{
  \begin{array}{ll}
    \frac{c(e_j)}{m(e_{i}^{\pm})}\;\;\;\;\;\;\;\hbox{ if } e_{j}^{\pm}=e_{i}^{\pm}\\
    \\
   - \frac{c(e_j)}{m(e_{i}^{\mp})}\;\;\;\;\hbox{ if } e_{j}^{\pm}=e_{i}^{\mp}\\
    \\
   0\;\;\;\;\;\;\;\hbox{ otherwise }.
\end{array}
\right.$$
\medskip
\bigskip
\\
Moreover the proofs are the same except for Proposition \ref{zero}. If
the graph contains a cycle $C_\ell=(e_1,...,e_\ell)$, we consider the 1-form
$$
\varphi(e)=\left\{
  \begin{array}{ll}
    \frac{1}{c(e)}\;\;\;\;\;\;\;\hbox{ if } e \in C_\ell\\
    \\
   0\;\;\;\;\;\;\;\hbox{ otherwise }.
  \end{array}\right.
  $$
  Then, if $x=e_i^+=e_{i+1}^-$ for $i \in \lbrace 1,...,\ell-1\rbrace$ (with the convention
  $e_{\ell^{+}}=e_1^{-}$) wich gives that
  \[\delta\varphi (x)=\frac{1}{m(x)}(1-1)=0\]
  and if $x$ does not belong to the cycle then $\delta\varphi (x)=0$ also. So we have the following theorem:
  \begin{Theo}Let $(G,m,c)$ be a weighted finite graph, then
    \[\sigma_p(\Delta_0)\setminus\{0\}=\sigma_p(\Delta_1)\setminus\{0\}.\]
    Moreover the multiplicity of $0\in\sigma_p(\Delta_0)$ equals the number of
    connected components of the graph; if $G$ possesses a cycle then $0\in\sigma_p(\Delta_1)$
    and if the graph is connected then the multiplicity of $0\in\sigma_p(\Delta_1)$
    equals the circuit rank of the underlying combinatorial graph.
    \end{Theo}
\section{Infinite case}\label{infinite}
The purpose of this part is to extend the previous results to the infinite weighted graph and to find the relation between the spectrum of $\Delta_0$ and $\Delta_1$. We are interested in the study of the problem of 0-spectrum. 
\subsection{Weyl's  criterion}
We consider the unbounded and self-adjoint operator on a Hilbert space, then we can use Weyl's criterion to characterize its spectrum like in \cite{Ay} and \cite{HS}.\\
\textbf{\underline{\emph{Weyl's criterion :}}} Let $\mathcal{H}$ be a separable Hilbert space, and let $\Delta$ be an unbounded self-adjoint operator on $\mathcal{H}$. Then
$\lambda$ is in  the spectrum of $\Delta$ if and only if there exists a sequence $(f_{n} )_{n\in \N}$ so that $\norm{f_{n}}=1$ and $\lim\limits_{n \to \infty} \norm {(\Delta-\lambda)f_{n}} = 0.$

We denote $\sigma(\Delta)$ the spectrum of $\Delta$ and we set
\begin{itemize}
\item $\sigma_{d}(\Delta)$ is the set of $\lambda \in \sigma(\Delta)$ which is an isolated point and an eigenvalue with finite multiplicity.\\
\item $\sigma_{ess} (\Delta ):=\sigma(\Delta) \setminus \sigma_{d}(\Delta)$.
\end{itemize}
\textbf{\underline{\emph{Weyl's criterion for essential spectrum :}} }Let $\Delta$ be an unbounded self-adjoint operator. Then we have $\lambda \in \sigma_{ess} (\Delta )$ if and only if there exists a sequence $(f_{n} )_{n\in \N} \in D(\Delta)$ so that $\norm{f_{n}}=1$, $(f_{n} )_{n\in \N}$ \textit{weakly converges to zero} and $\lim\limits_{n \to \infty} \norm {(\Delta-\lambda)f_{n}} = 0.$
\begin{Prop} If the operators $\Delta_0$ and $\Delta_1$ are essentially selfadjoint then
$$\sigma_{ess}(\Delta_{1})\setminus \{0\}= \sigma_{ess}  (\Delta_{0})\setminus \{0\}.$$
\end{Prop}
\begin{Demo} Denote $\mathcal D(\Delta_0)$, resp. $\mathcal D(\Delta_1)$, the domains of the
  selfadjoint extensions of $\Delta_0$ resp. $\Delta_1$.
  
  Let $\lambda \neq 0$ in the essential spectrum of $\Delta_{0}$ then, by Weyl's criterion, there
  exists a sequence $(f_{n} )_{n\in \N}$ in $\mathcal D(\Delta_0)$ such that
  $$n\in \N, ~~ \norm{f_{n}}_{\mathcal{V}}=1 ,\; \lim\limits_{n \to \infty}\langle f_n,f \rangle_{\mathcal{V}} =0,\;\forall f \in l^{2}(\mathcal{V})  \hbox{ and } \lim\limits_{n \to \infty}\norm{(\Delta_{0}-\lambda)f_{n}}_{\mathcal{V}}=0.$$ So, we want to find a sequence $(\varphi_{n})_{n}$ of $\mathcal{C}_{0}(\mathcal{E})$ such that 
$$n\in \N, ~~\norm{\varphi_{n}}_{\mathcal{E}}= 1 \; \lim\limits_{n \to \infty}\langle \varphi_n,\varphi \rangle_{\mathcal{E}} =0,\;\forall \varphi \in l^{2}(\mathcal{E})\hbox{ and } \lim\limits_{n \to \infty} \norm{(\Delta_{1}-\lambda)\varphi_{n}}_{\mathcal{E}}=0.$$

As $\Delta_0$ is essentially selfadjoint , $\mathrm{ d }f\in l^2(\Ec)$ and we set
$$\varphi_{n}:=\frac{\mathrm{ d }f_{n}}{\norm{\mathrm{ d }f_{n}}_{\mathcal{E}}}.$$

$\bullet$ So $\|\varphi_n\|_{\mathcal{E}}=1, ~~n\in\N$, we first prove that $\varphi_n,$ converges weakly to 0.

(1) We see that $\langle(\Delta_{0}-\lambda)f_n,f_n\rangle_{\mathcal{V}}=\|\mathrm{ d }f_n\|_{\mathcal{E}}^2 -\lambda \to 0$ when $n\to \infty$, by
hypothesis on the sequence $ (f_{n} )_{n\in \N}$. But $\lambda\neq 0$ so there exist $n_0$ and two positive constants
$\alpha$ ,$\beta$ such that
\[\forall n\geq n_0 ,~\quad 0<\alpha<\|\mathrm{ d }f_n\|_{\mathcal{E}}<\beta.\]

(2) Let $\varphi\in \mathcal D(\Delta_1)$ then, for any $n\geq n_0$
\begin{align*}
  \langle \varphi_n,\varphi\rangle_{\mathcal{E}}&=\frac{1}{\|\mathrm{ d }f_n\|}\langle f_n,\delta\varphi\rangle_{\mathcal{V}}\\
  &\leq \frac{1}{\alpha}\langle f_n,\delta\varphi\rangle_{\mathcal{V}}\overset{n\to\infty}\longrightarrow 0.
\end{align*}

(3) for any $\varphi \in l^{2}(\mathcal{E})$ and any $\eps>0$, as $\mathcal D(\Delta_1)$
is dense in $l^{2}(\mathcal{E})$ there exits $\bar\varphi\in\mathcal D(\Delta_1)$ such that 
\begin{align*}\|\varphi-\bar\varphi\|_{\mathcal{E}}&\leq \eps/2\Rightarrow\\
  |\langle\varphi_n,\varphi\rangle_{\mathcal{E}}|&\leq |\langle\varphi_n,\bar\varphi\rangle_{\mathcal{E}}|+
  |\langle\varphi_n,\varphi-\bar\varphi\rangle_{\mathcal{E}}|\\
  &\leq |\langle\varphi_n,\bar\varphi\rangle_{\mathcal{E}}|+\eps/2.
\end{align*}
But there exists $n_1\geq n_0$ such that
\[\forall n\geq n_1\quad |\langle\varphi_n,\bar\varphi\rangle_{\mathcal{E}}|\leq \eps/2\]
by the application of (2). So we have proved:
\[\forall\eps>0,\; \exists\, n_1>0 \hbox{ such that }\; \forall n\geq n_1\quad
|\langle\varphi_n,\varphi\rangle_{\mathcal{E}}|\leq \eps.\]

$\bullet$ Finally we prove that we can choose $(f_n)$ such that moreover
$\lim_{n\to\infty}\|(\Delta_1-\lambda)\varphi_n\|_{\mathcal{E}}=0$.

Pick $\bar\lambda >\lambda$ and denote $\mathrm P_{\Omega}$, the
spectral projector of $\Delta_0$ relative to the borelian $\Omega$.
Applying the Spectral Theorem (Thm VIII.5 of \cite{RS}) we have
that
\[f_n=\mathrm P_{]-\infty,\bar\lambda[ }(f_n)+\mathrm P_{[\bar\lambda,+\infty[ }(f_n)\]
    an orthogonal decomposition and we know
    that $\Delta_0\geq \bar\lambda$ in $\Ran(P_{[\bar\lambda,+\infty[ })$.
        Then
        \begin{align*}
          \|(\Delta_0-\lambda)(f_n)\|_{\mathcal{V}}^2&= \|(\Delta_0-\lambda)\mathrm P_{]-\infty,\bar\lambda[ }(f_n)\|_{\mathcal{V}}^2+
          \|(\Delta_0-\lambda)\mathrm P_{[\bar\lambda,+\infty[ }(f_n)\|_{\mathcal{V}}^2\\
              &\geq \|(\Delta_0-\lambda)\mathrm P_{]-\infty,\bar\lambda[ }(f_n)\|_{\mathcal{V}}^2+(\bar\lambda-\lambda)_{\mathcal{V}}^2
              \|\mathrm P_{[\bar\lambda,+\infty[ }(f_n)\|_{\mathcal{V}}^2
        \end{align*}

        So we have
        \[\lim_{n\to\infty}\|\mathrm P_{[\bar\lambda,+\infty[ }(f_n)\|_{\mathcal{V}}=0
            \Rightarrow \lim_{n\to\infty}\|\mathrm P_{]-\infty,\bar\lambda[ }(f_n)\|_{\mathcal{V}}=1
            \]

            It means that we can work with
            $\bar f_n=\mathrm P_{]-\infty,\bar\lambda[ }(f_n)/\|\mathrm P_{]-\infty,\bar\lambda[ }(f_n)\|_{\mathcal{V}}$
            witch satisfies also
            \begin{equation*}
              \|\bar f_n\|_{\mathcal{V}}=1,\quad \bar f_n\overset{n\to\infty}\rightharpoonup 0\; \hbox{ weakly,}\quad
              \|(\Delta_0-\lambda) \bar f_n\|_{\mathcal{V}}\overset{n\to\infty}\longrightarrow 0.
            \end{equation*}

            Define now $\bar\varphi_n=\mathrm{ d }\bar f_n/\|\mathrm{ d }\bar f_n\|_{\mathcal{E}}$.
Then, as before, $\|\bar\varphi_n\|_{\mathcal{E}}=1$, $\bar \varphi_n\overset{n\to\infty}\rightharpoonup 0$
weakly, but now, as
\begin{equation}\forall f\in\Ran \mathrm P_{]-\infty,\bar\lambda[}, \quad
    \|\mathrm{ d }f\|_{\mathcal{E}}^2=\langle\Delta_0 f,f\rangle_{\mathcal{V}}\leq\|\Delta_0 f\|\,\|f\|_{\mathcal{V}}\leq\bar\lambda\|f\|_{\mathcal{V}}^2
\end{equation}
and also $(\Delta_0-\lambda)\bar f_n\in\Ran \mathrm P_{]-\infty,\bar\lambda[}$ we
have
  \[\|(\Delta_1-\lambda)\bar\varphi_n\|_{\mathcal{E}}=\frac{\|\mathrm{ d }(\Delta_0-\lambda)\bar f_n\|_{\mathcal{E}}}{\|\mathrm{ d }\bar f_n\|_{\mathcal{E}}}
  \leq\sqrt{\bar\lambda}\frac{\|(\Delta_0-\lambda)\bar f_n\|_{\mathcal{V}}}{\|\mathrm{ d }\bar f_n\|_{\mathcal{E}}}
  \leq\frac{\sqrt{\bar\lambda}}{\alpha}\|(\Delta_0-\lambda)\bar f_n\|_{\mathcal{V}}\overset{n\to\infty}\longrightarrow 0.\]

  If now $\mu\neq 0$ in the essential spectrum of $\Delta_{1}$ we can do the
  same thing inverting the roles of $\mathrm{ d }$ and $\delta$.

\end{Demo}

\begin{Prop}
Let $G$ be a weighted, connected, locally finite graph containing at least one cycle, then $0$ is an eigenvalue of $\Delta_1$.
\end{Prop}
\begin{Demo}
Note $C_n=(e_1,e_2,...,e_n)$ a cycle of $G$.\\
 Let $\varphi$ be a function in $\mathcal{C}^{a}_0(\mathcal{E})$ defined by
$$\varphi(e)=\left\lbrace \begin{array}{cc}
\dfrac{1}{c(e)}&\text{ if } e=e_i, ~i=1,..,n\\
0 &\text{ otherwise }
\end{array}\right. $$
 We follow the same approach as that of the proof of proposition \ref{zero} fo find that $\Delta_1\varphi (e)=0$ for all $e\in\mathcal{E}$.
\end{Demo}

\subsection{Spectrum on $\boldsymbol\chi$-complete Graph}
In the unbounded case we need at least some completeness.
In \cite{AT} we have introduced the notion of $\chi$-completeness. Let us 
recall the definition.
\begin{defi}The graph $G$ is {\em $\chi-$complete}
if there exists an increasing sequence of finite sets 
$(B_n)_{n\in \N}$ such that
$\mathcal{V} =\cup_{n} B_n$ and there exist related functions $\chi_n$ satisfying the
following three conditions:
\begin{enumerate}
\item  $\chi_n\in \mathcal{C}_0(\mathcal{V} ),\, 0\leq\chi_n\leq 1$
\item  $v\in B_n ~\Rightarrow ~\chi_n(v)=1$
\item  $\displaystyle\exists C>0, \forall n\in \N,\, x\in \mathcal{V}\; ,
\frac{1}{m(x)}\sum_{e,e^\pm=x}c(e)\di\chi_n(e)^2\leq C.$
\end{enumerate}
\end{defi}
We know by \cite{AT} that on a $\chi-$complete graph the Gau\ss-Bonnet 
operator, and then the Laplacian is essentially self adjoint, let
$\Delta=\Delta_0\oplus \Delta_1$ its unique extension. Moreover it was proved in
\cite{LSW21} that $\chi-$completeness is equivalent to a geometric property : the
existence of an intrinsic pseudo-metric with finite balls.
\begin{Prop}Let $G=(\mathcal{V},\mathcal{E})$ be a connected, locally finite $\chi-$complete 
graph. Suppose that $0\notin\sigma(\Delta_0)$, then
$$\inf(\sigma(\Delta_1)\backslash\{0\})=\inf\sigma(\Delta_0)>0.$$
\end{Prop}
\begin{proof}
As the graph is $\chi-$complete, $\di_{min}=\di_{max}$ and 
$\delta_{min}=\delta_{max}.$ As $\Delta_0$ is invertible we have
$\Ker\di=\{0\}$ and thus the decomposition
\begin{equation*}l^2(\mathcal{V})=\overline{\Im\delta} \quad 
l^2(\mathcal{E})=\Ker\delta\oplus \overline{\Im\di}.
\end{equation*}
But $\Delta_0$ invertible tells us more :
\begin{equation*}
\exists h>0,\;\forall f\in C_0(\mathcal{V})\quad
\|\Delta_0(f)\|_{\mathcal{V}}\geq h\|f\|_{\mathcal{V}}.
\end{equation*}
Thus $\inf\sigma(\Delta_0)>0$. Moreover, by the spectral theorem, 
$\sqrt h$ is a lower bound of $\sqrt\Delta_0,$
but $$\|\sqrt\Delta_0(f)\|_{\mathcal{V}}^2=\langle \Delta_0(f),f \rangle_{\mathcal{V}}=\|\di(f)\|_{\mathcal{E}}^2.$$ It means that
\begin{equation}\label{cheeger}
\forall f\in C_0(\mathcal{V})\quad
\|\di(f)\|_{\mathcal{E}}^2\geq h\|f\|_{\mathcal{V}}^2
\end{equation}
and, indeed, this inequality is true for any function in the domain
of $\di_{max}.$

Now let $\varphi\in(\Ker\delta)^\perp\cap\delta$ with
$\|\varphi\|_{\mathcal{E}}=1.$  As $C_0(\mathcal{V})$ is dense in $l^2(\mathcal{V})$
$$\|\delta(\varphi)\|_{\mathcal{V}}=\sup_{f\in C_0(\mathcal{V}),\|f\|_{\mathcal{V}}=1}\langle\delta(\varphi),f\rangle_{\mathcal{V}}
=\sup_{f\in C_0(\mathcal{V}),\|f\|_{\mathcal{V}}=1}\langle\varphi,\di(f)\rangle_{\mathcal{E}}$$
But, as $\varphi\in\overline{\Im(\di)}$
$$\exists(f_n)_{n\in\N}, f_n\in C_0(\mathcal{V})\quad\lim_{n\to\infty}\|\di(f_n)-\varphi\|_{\mathcal{E}}=0.
$$ 

We have then
\begin{equation}\forall n\in\N,\; \langle\varphi,\di(f_n)\rangle_{\mathcal{E}}\leq\|\delta(\varphi)\|_{\mathcal{V}}\,\|f_n\|_{\mathcal{V}}.
\end{equation}
We apply now the inequality (\ref{cheeger}) to $f_n$
\begin{equation}\forall n\in\N,\; \langle\varphi,\di(f_n)\rangle_{\mathcal{E}}\leq\|\delta(\varphi)\|_{\mathcal{V}}\,
\frac{\|\di(f_n)\|_{\mathcal{E}}}{\sqrt h}
\end{equation}
and thus, at the limit
\begin{equation} \langle\varphi,\varphi\rangle_{\mathcal{E}}\leq\|\delta(\varphi)\|_{\mathcal{E}}\,
\frac{\|\varphi\|_{\mathcal{E}}}{\sqrt h}.
\end{equation}

We have proved that $h$ is also a lower bound of $\sigma(\Delta_1)\backslash\{0\}.$
If we take $h=\inf\sigma(\Delta_0),$ it proves that
$$\inf(\sigma(\Delta_1)\backslash\{0\})\geq\inf\sigma(\Delta_0)>0.$$
Using that $l^2(\mathcal{V})=\overline{\Im\delta}$ one can show in the same way
$$\inf(\sigma(\Delta_1)\backslash\{0\})\leq\inf\sigma(\Delta_0).$$
\end{proof}
 
\begin{Corol}\label{vp}
Let $G=(\mathcal{V},\mathcal{E})$ be a connected, locally finite $\chi-$complete 
graph. Suppose that $0\notin\sigma(\Delta_0),$ then if  
$0\in\sigma(\Delta_1),$ then it is an isolated
value  and, as a consequence,  it is an eigenvalue (maybe with infinite 
multiplicity).
\end{Corol}
This is a direct application of the Theorem VII.11 of \cite{RS}.

On the other hand we have the following generalization of \cite{Ay}
\begin{Theo}Let $G=(\mathcal{V},\mathcal{E})$ be a connected, locally finite $\chi-$complete 
  weighted graph. We suppose that the weights $(m,c)$ satisfy the following properties :
  \begin{align*}
    (1)\; & \exists M>0,\,\forall x\in\Vc \; \sum_{y\sim x} c(x,y)\leq M m(x)\\
          (2)\; &\exists \gamma>0,\, \forall (x,y)\in\Ec, \; c(x,y)\geq \gamma
  \end{align*}
then $0\in\sigma(\Delta_0)$ or $0\in\sigma(\Delta_1)$.
\end{Theo}
\begin{rema}
The hypothesis (1) assures that the Laplacians are bounded and if $0\notin\sigma(\Delta_0)$ then the Cheeger constant is positive and the Cheeger inequality holds. With (1) and (2) the proof of lemma 3.5 of \cite{Ay} still true.
\end{rema}
\subsection{Examples}\label{exp}
The idea of this examples comes from \cite{ABT1} and \cite{ABT2}.
\begin{Exp}We recall that on the simple graph $(\Vc=\Z,\Ec=\{(x,y), |x-y|=1\})$
  the spectra are purely continuous :
  \[\sigma(\Delta_0)=\sigma(\Delta_1)=[0,4].\]
\end{Exp}
\begin{Exp}
Let us consider the graph $G_1$ with $\mathcal{V}_1 =\Z$, $\mathcal{E}_1=\{(x,y), |x-y|=1\},$
conductivity $c(j,j+1)=(j+1)^2+1$ and
voltage mass constant equal to 1.
\begin{Prop}The graph $G_1$ is $\chi-$ complete.
\end{Prop}
Indeed, let, for $n\in\N\backslash\{0\}$
\begin{equation*}
\chi_n(x)=\Big(2-\frac {|x| }{n+1}\vee 0\Big)\wedge 1
\end{equation*}
Then $\chi_n(x)=1$ on $[-(n+1),n+1]\cap\Z$, $\chi_n(x)=0$ outside of 
$[-2(n+1),2(n+1)]\cap\Z$.\\

 We remark that if $y\sim x$ then, by the triangle inequality, $||x|-|y||\leq 1$ so
\begin {equation}\label{eq:dchi}
  y\sim x\Rightarrow |\mathrm{ d }^{0}\chi_{n}(x,y)|=|\chi_{n}(x)-\chi_{n}(y)|\leq\frac{1}{n+1}.
\end {equation}
Thus, $\forall n\in \N\backslash\{0\},\, x \in \mathcal{V}_1$
\begin{equation*}
\frac{1}{m(x)}\sum_{e,e^\pm=x}c(e)\di\chi_n(e)^2\leq \sup_{|j|\leq (n+1)+2{n}}
\frac{2(|j|^2+1)}{({n+1})^2}\leq\frac{4(9n^2+1)}{n^2}\leq 40.
\end{equation*}
\begin{Prop}
$0\notin\sigma(\Delta_0).$
\end{Prop}


Let $f\in\mathcal{C}_0(\mathcal{V})$, we have
\begin{align*}
\|f\|_{\mathcal{V}}^2=&\sum_{k\in \mathbb{Z}}|f(k)|^2\\
=&\sum_{k\in \mathbb{Z}}(k+1-k)|f(k)|^2\\
=&\sum_{k\in \mathbb{Z}}(k+1)|f(k)|^2-\sum_{k\in \mathbb{Z}}k|f(k)|^2\\
=&\sum_{k\in \mathbb{Z}}k(|f(k-1)|^2-|f(k)|^2)\\
=&\sum_{k\in \mathbb{Z}}k(|f(k-1)|-|f(k)|)(|f(k-1)|+|f(k)|).
\end{align*}
From Cauchy-Schwarz inequality, we obtain
\begin{align*}
\|f\|_{\mathcal{V}}^2\leq &\Big(\sum_{k\in \mathbb{Z}}k^2|f(k-1)-f(k)|^2\Big)^\frac{1}{2}\Big(\sum_{k\in \mathbb{Z}}(|f(k-1)|+|f(k)|)^2\Big)^\frac{1}{2}\\
\leq &\Big(\sum_{k\in \mathbb{Z}}k^2|f(k-1)-f(k)|^2\Big)^\frac{1}{2}\Big(4\sum_{k\in \mathbb{Z}}|f(k)|^2\Big)^\frac{1}{2}
\end{align*}
Hence
$$\dfrac{1}{4}\|f\|_{\mathcal{V}}^2\leq \sum_{k\in \mathbb{Z}}(k^2+1)|f(k-1)-f(k)|^2.$$
It follows that
$$\dfrac{1}{4}\leq\dfrac{\langle\Delta_0 f,f\rangle_{\mathcal{V}}}{\langle f,f\rangle_{\mathcal{V}}}=\dfrac{\displaystyle{\sum_{k\in \mathbb{Z}}(k^2+1)|f(k-1)-f(k)|^2}}{\langle f,f \rangle_{\mathcal{V}}}.$$
\begin{Prop}
$0\in\sigma(\Delta_1).$
\end{Prop}
We recall that $\ker(\Delta_1)=\ker(\delta)$ and we want to construct an eigenfunction $\varphi \neq 0$ such that $\Delta_1(\varphi)=0$, which gives $\delta(\varphi)=0$. We define
$$\varphi_{n+1}=\varphi(n,n+1),\; n\in \Z.$$ We have $$\delta(\varphi)(n)=((n+1)^2+1)\varphi_{n+1}-(n^2+1)\varphi_{n}=0 .$$
So, the 1-form $\varphi$ must satisfy 
\begin{multline*}
\forall n\in\N,\, ((n+1)^2+1)\varphi_{n+1}=(n^2+1)\varphi_{n}=\dots=\varphi_0\Rightarrow 
\varphi_{n+1}=\frac{1}{(n+1)^2+1}\varphi_0\\
\hbox{and in the same way}\quad \forall n\in\N,\,\varphi_{-n}=\frac{1}{n^2+1}\varphi_0\\
\Rightarrow\|\varphi\|_{\mathcal{E}}^2 = \sum_{n\in\Z}|\varphi_0|^2((n)^2+1)\frac{1}{((n)^2+1)^2} = |\varphi_0|^2\sum_{n\in\Z}\frac{1}{(n)^2+1} .
\end{multline*}
Then we obtain, $\varphi\in l^2(\mathcal{E})$.
\end{Exp}
Moreover Proposition 3.5 can be generalized as follows
\begin{Prop}\label{ex3.2}Let $G=\Z$ with the weight on vertices $m=1$ and the weight $c$ on edges. If there exists a function $A$ on $\mathcal E$ and a constant $a>0$ such that
  \[\forall n\in\Z,\; A(n,n+1)-A(n-1,n)\geq a \hbox{ and }A^2(n,n+1)=c(n,n+1)
  \]
  then \[\sigma(\Delta_0)\subset [\frac{a^2}{4},+\infty).\]
\end{Prop}
\begin{proof}Let $f$ be any function defined on $\Z$ with finite support
  \begin{align*}
    a\|f\|^2=a\sum_{n\in\Z}f(n)^2&\leq\sum_{n\in\Z}(A(n,n+1)-A(n-1,n))f(n)^2\\
    &\leq \sum_{n\in\Z}A(n,n+1)(f(n)^2-f(n+1)^2)\\
    &\leq \sum_{n\in\Z}A(n,n+1)(f(n)-f(n+1))(f(n)+f(n+1))\\
    &\leq \sqrt{\sum_{n\in\Z}c(n,n+1)(f(n)-f(n+1))^2}\sqrt{\sum_{n\in\Z}(f(n)+f(n+1))^2}\\
    &\leq2 \|\mathrm{ d }f\|\|f\|
  \end{align*}
  But $\Delta_0$ is essentially self-adjoint in this case by \cite{nth}, the conclusion follows then by
  the minimax formula.
\end{proof}
In addition, Theorem 14.29 of \cite{JP} tells that on the graph $\Z$ with constant weight
on vertices, and weight $c$ on edges the existance of an harmonic function with finite energy is equivalent to the condition $\displaystyle \sum_{e} \frac{1}{c(e)}$ is finite.
We list in the following few results concerning this specific situation.

\begin{Prop}\label{14.29}Let $G=\Z$ with the weight on vertices $m=1$ and the weight $c$ on edges. Then we have 
  $$ \sum_{e} \frac{1}{c(e)} < \infty                   \Longleftrightarrow  \ker(\delta_1)\neq\{0\} \hbox{ on }
  l^2(\mathcal E).$$
\end{Prop}
\begin{proof}[Proof of the proposition]
 
  We start by the first implication, supposing that the $(\frac{1}{c(n,n+1)})_{n\in\Z}$ is sommable
  we have to find $\varphi \in l^2(\Ec)$, $\varphi \neq 0$ and $\delta_{1} \varphi=0$.\\

We set $\varphi (e)=\frac{1}{c(e)}$, then we obtain 
\begin{multline*}
\delta_{1} \varphi(n)= c(n-1,n) \varphi(n-1,n) -c(n,n+1) \varphi(n,n+1)\\= c(n-1,n)\frac{1}{c(n-1,n)} -c(n,n+1)\frac{1}{c(n,n+1)}= 1-1=0.
\end{multline*}
And we have 
$$\norm{\varphi}_{\mathcal{E}}^{2}=  \sum_{e} {c(e)} \varphi(e)^2 =  \sum_{e} {c(e)} \left( \frac{1}{c(e)}\right) ^2=\sum_{e} \frac{1}{c(e)} < \infty  $$
so the form $\varphi$ is in $l^2(\Ec)$.

Now, we suppose that there exists $\varphi \in l^2(\Ec)$, $\varphi \neq 0$ satisfying $\delta_{1} \varphi=0$.
We obtain 
 $$c(n-1,n) \varphi(n-1,n) -c(n,n+1) \varphi(n,n+1)= 0$$ witch gives
 
$$c(n-1,n) \varphi(n-1,n) =c(n,n+1) \varphi(n,n+1)= a \in \R$$

We know that $\norm{\varphi}_{\mathcal{E}}<\infty$, but
\begin{equation*}
 \norm{\varphi}_{\mathcal{E}}^2=\sum_{e}c(e)\left(a\frac{1}{c(e)}\right)^2=a^2 \sum_{e} \frac{1}{c(e)}  
\end{equation*}
So, we have $a\neq 0$ and $\sum_{e} \frac{1}{c(e)}<\infty$.
\end{proof}
\begin{rema}The harmonic function given in \cite{JP} is then the integral of our $\varphi$.
\end{rema}
\begin{Exp}
Thanks to \cite{JP}, we consider the graph $\Vc=\Z$ with weights $m(n)=1$ and $c(n,n+1)=\alpha^{|n|},\;\alpha > 1$. \\
\begin{itemize}
\item By Proposition \ref{ex3.2}, 0 is not an eigenvalue of $\Delta_{0}$.
 Indeed, we take
 \begin{equation*}
    A(n,n+1)=\begin{cases}\alpha^{n/2} \hbox{ if $n\geq 0$}\\
    -\alpha^{-n/2}\hbox{ if $n\leq -1$}
    \end{cases}
  \end{equation*}
  For $n\geq 1$ we have
  \[\alpha^{n/2}-\alpha^{(n-1)/2}=\alpha^{(n-1)/2}(\sqrt\alpha-1)\geq(\sqrt\alpha-1)>0.\]
  For $n=0$ we have
  \[1-(-\sqrt\alpha)=1+\sqrt\alpha\geq(\sqrt\alpha-1).\]
  For $n\leq -1$ we have
  \[-\alpha^{-n/2}+\alpha^{(-n+1)/2}=\alpha^{-n/2}(\alpha^{1/2}-1)\geq(\sqrt\alpha-1).\]

  So the proposition applies with $a=(\sqrt\alpha-1)$.\\
  \item By Proposition \ref{14.29} we have, 0 is an eigenvalue of $\Delta_{1}$  because the following is satisfied
$$ \sum_{e} \frac{1}{c(e)} < \infty$$   \end{itemize}

  \hfill$\diamondsuit$
  \\
\end{Exp}
\begin{Exp}\label{gg}
Let us consider the graph $G_{2}$ with $\mathcal{V}_2=\Z$, $\mathcal{E}_2=\{(x,y), |x-y|=1 \text{ or } x=-y\},$
conductivity $c(j,j+1)=(j+1)^2+1$, $c(j,-j)=1$ and the voltage mass $m$ equal to 1.
  Let us consider the following infinite weighted graph $G_{2}$, see Figure 2, with (almost) constant degree.
  We denote the origin by $x_0=0$ and  by $S_n$ the spheres for the combinatoric distance of the symmetric underlying graph:
\[d_{\rm comb}(x_0,x)=\inf\{k;\,\exists \gamma=(x_0,\dots,x_k) \hbox{ a chain such that }x_k=x\}\]
So $S_n=\{x\in V,~ d_{\rm comb}(x_0,x)=n\}=\{n,-n\}$.\\

 \begin{figure}[ht]
\includegraphics[height=4cm,width=10cm]{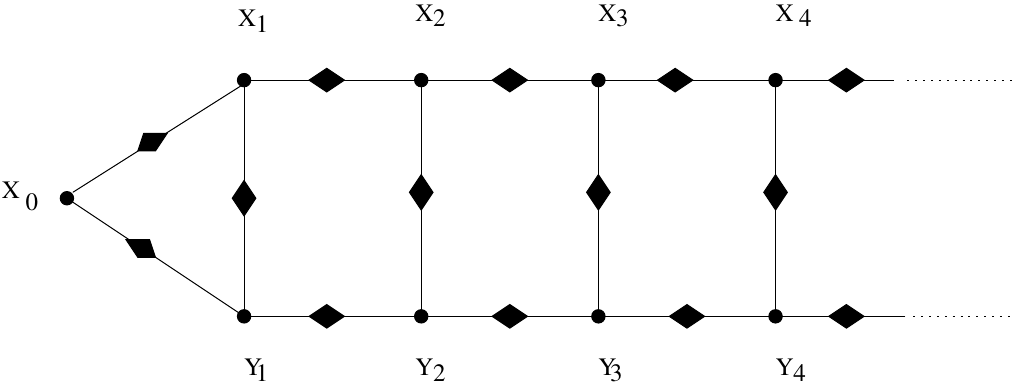}
\caption{A graph with almost constant degree}
 \end{figure}

\begin{Prop}
The graph $G_2$ is $\chi$-complete. And we have $$0\notin \sigma(\Delta_0) \hbox{ and } 0\in\sigma(\Delta_1).$$
\end{Prop}
\begin{Demo}
\begin{enumerate}
\item First, we show  that $G_2$ is $\chi$-complete from the criterion given
in Theorem 5.11 of \cite{BGJ}.
We remark that the sets $S_n^+$ and
$S_n^-$ introduced in \cite{BGJ} coincide with $S_n$ and for $x\in
S_n$ the weighted degree is constant and for $n\geq 2$:
\begin{gather*}
  a_n^+=\sup_{x\in S_n }\frac{1}{m(x)}\left(\sum_{y\in S_{n+1}}c(x,y)\right)=(n+1)^2+1\hbox{ , }
  a_n^-=\sup_{x\in S_n }\frac{1}{m(x)}\left(\sum_{y\in S_{n-1}}c(x,y)\right)=n^{2}+1\\
\Rightarrow
  \sum_{n=0}^\infty\dfrac{1}{\sqrt{a_n^++a_{n+1}^-}}\geq\sum_{n=2}^\infty\dfrac{1}{\sqrt{2}(n+2)}=\infty.
\end{gather*}
Thus, this graph satisfies the hypothesis of \cite{BGJ}.\\
\item Next, we show that $0\notin \sigma(\Delta_0)$. Indeed, we see that the graph $G_2$ is obtained by adding the edges to the graph $G_1$ considered in the previous example, then
\begin{align*}
\langle\Delta_0 f,f\rangle_{\mathcal{V}}= &\Big(\sum_{(x,y)\in\mathcal{E}_2}c(x,y)|f(x)-f(y)|^2\Big)\\
=& \sum_{(x,y)\in\mathcal{E}_1}c(x,y)|f(x)-f(y)|^2+\sum_{(k,-k)\atop k\in \mathbb{Z}}|f(k)-f(-k)|^2\\
=&\sum_{k\in\mathbb{Z}}(k^2+1)|f(k-1)-f(k)|^2+\sum_{(k,-k)\atop k\in \mathbb{Z}}|f(k)-f(-k)|^2
\end{align*}
this implies that
$$\dfrac{1}{4}\leq\dfrac{\displaystyle{\sum_{k\in \mathbb{Z}}(k^2+1)|f(k-1)-f(k)|^2}}{\langle f,f \rangle_{\mathcal{V}}}\leq \dfrac{\langle \Delta_0 f,f \rangle_{\mathcal{V}}}{\langle f,f \rangle_{\mathcal{V}}}.$$
\item Finally, We remark that this graph contains an infinity number of cycles, then we can construct a sequence of independent eigenfunctions $f_k$ corresponding to 0 in the spectrum of $\Delta_1$, acting on the independent cycles $\left\lbrace C_{2k+1}\right\rbrace _{k \geq 0}$ where $$C_1=\left((0,1), (1,-1), (-1,0) \right) \text{ and } C_k=\left( ( k-1,k), ( k,-k),(-k,-k+1),( -k+1,k-1)  \right).$$ 
We define $\varphi_k$ by 
$$\varphi_k(e)=\left\lbrace \begin{array}{cc}
\dfrac{1}{c(e)}&\text{ if } e\in C_k \\
0 &\text{ otherwise }
\end{array}\right. $$
It follows that $0$ is an eigenvalue  of $\Delta_1$ with infinite multiplicity. Hence $0$ is in the essential spectrum.\\
\end{enumerate}
\end{Demo}
\end{Exp}

\textbf{\textit{\emph{Acknowledgments}}:}
All the authors would like to thank the anonymous referee for careful reading, numerous remarks, useful suggestions and valuable reference.

\end{document}